\documentclass[12pt]{amsart}
\usepackage{amscd,amssymb, latexsym}
\usepackage{amsmath,amsthm,amstext, amsbsy, amssymb}
\usepackage{epsfig,graphics,graphicx,mathrsfs}
\newcommand{\ov}{\overline}

\DeclareMathOperator{\diam}{diam}
\DeclareMathOperator{\dist}{dist}
\DeclareMathOperator{\sign}{sgn}
\DeclareMathOperator{\conv}{conv}
\DeclareMathOperator{\ND}{\mathit{ND}}
\DeclareMathOperator{\Nev}{\mathit{N}}

\let\Re\xRe
\let\Im\xIm
\newtheorem{theorem}{Theorem}
\newtheorem{lemma}{Lemma}[section]

\newtheorem{proposition}{Proposition}
\newtheorem{definition}{Definition}

\theoremstyle{remark}
\newtheorem{remark}[theorem]{Remark}

\makeatletter
\@addtoreset{equation}{section}
\makeatother

\begin{document}

\title{Nevanlinna domains with large boundaries}

\author[Yu.~Belov, A.~Borichev, K.~Fedorovskiy]{Yurii Belov, Alexander Borichev, Konstantin Fedorovskiy}

\begin{abstract}
We establish the existence of Nevanlinna domains with large boundaries. In particular, these domains can have boundaries
of positive planar measure. The sets of accessible points can be of any Hausdorff dimension between $1$ and $2$.
As a quantitative counterpart of these results, we construct rational functions univalent in the unit disc with
extremely long boundaries for a given amount of poles.
\end{abstract}

\thanks{Theorems~1 and 3 are obtained in the framework of
the project 17-11-01064 by the Russian Science Foundation. The remaining part
of this research (Theorems 2 and 4) was partially supported by a joint grant
of Russian Foundation for Basic Research (project 17-51-150005-NCNI-a) and
CNRS (project PRC CNRS/RFBR 2017-2019 ``Noyaux reproduisants dans des espaces
de Hilbert de fonctions analytiques''), and by Ministry of Education and
Science of the Russian Federation (project 1.517.2016/1.4).}

\address{
\hskip -\parindent Yurii Belov:
\newline \indent Saint Petersburg State University, St.~Petersburg, Russia,
\newline \indent {\tt j\_b\_juri\_belov@mail.ru}
\smallskip
\newline \noindent Alexander Borichev:
\newline \indent Aix--Marseille University, CNRS, Centrale Marseille, I2M,
\newline \indent {\tt alexander.borichev@math.cnrs.fr}
\smallskip
\newline \noindent Konstantin Fedorovskiy:
\newline \indent Bauman Moscow State Technical Univrsity, Moscow, Russia,
\newline \indent Saint Petersburg State University, St.~Petersburg, Russia,
\newline \indent {\tt kfedorovs@yandex.ru}
}

\maketitle

\section{Introduction}

Nevanlinna domains constitute an interesting class of
bounded simply connected domains in the complex plane $\mathbb C$. They play a
crucial role in recent progress in problems of uniform approximation of
functions on compact sets in $\mathbb C$ by polynomial solutions of elliptic
equations with constant complex coefficients. In this paper we give a complete
solution to the following problem posed in the early \text{2000-s:} how large (in the
sense of dimension) can be the boundaries of Nevanlinna domains?

\subsection{Nevanlinna domains}


Denote by $\mathbb D$ the open unit disc $\{z\in\mathbb C\colon |z|<1\}$ and
let $\mathbb T=\partial\mathbb D$ be the unit circle. For an open set
$U\subset\widehat{\mathbb C}$ let us denote by $H^{\infty}(U)$ the set of
all bounded holomorphic functions on $U$.
%

\begin{definition}[\rm see~\cite{CarFedPar2002msb}, Definition~2.1]
A bounded simply connected domain $G\subset\mathbb C$ is a
Nevanlinna domain if there exist two functions $u,v\in H^{\infty}(G)$ with
$v\not\equiv0$ such that the equality
\begin{equation}\label{nd1}
\overline{z}=\frac{u(z)}{v(z)}
\end{equation}
holds on $\partial G$ almost everywhere in the sense of conformal mappings.
%
\end{definition}

Property \eqref{nd1} means the equality of angular boundary
values
\begin{equation}\label{nd2}
\overline{f(\zeta)}=\frac{(u\circ f)(\zeta)}{(v\circ f)(\zeta)}
\end{equation}
for almost all $\zeta\in\mathbb T$, where $f$ is a conformal mapping
from $\mathbb D$ onto $G$. We recall here that for every function
$f\in H^{\infty}(\mathbb D)$ and for almost all (with respect to Lebesgue
measure on $\mathbb T$) points $\zeta\in\mathbb T$ there exists the finite
angular boundary value $f(\zeta)$.

For the sake of brevity, we call Nevanlinna domains $\Nev$-domains, and we denote by
$\ND$ the class of all Nevanlinna domains.

Note that the definition of a Nevanlinna domain does not depend on the
choice of $f$. Moreover, in view of the Luzin--Privalov boundary uniqueness
theorem, the quotient $u/v$ is uniquely defined in $G$ (for a Nevanlinna
domain $G$). If $G$ is a Jordan domain with rectifiable boundary, then the
equality~\eqref{nd1} may be understood directly as the equality of angular
boundary values almost everywhere with respect to the Lebesgue measure on
$\partial G$. The equality~\eqref{nd1} can be similarly understood on any
rectifiable Jordan arc $\gamma\subset\partial G$ such that each point
$a\in\gamma$ is not a limit point for the set $\partial G\setminus\gamma$.
Note that for Jordan domains with rectifiable boundaries the concept of a
Nevanlinna domain was introduced in \cite{Fed1996mn} in slightly different
terms.


It can be readily verified that every disc is a Nevanlinna domain, while every
domain which is bounded by an ellipse which is not a circle, or by a
polygonal line is not in $\ND$. Yet another interesting example of a
Nevanlinna domain is Neumann's oval, i.e. the domain bounded by the image
of an ellipse (which is not a circle) with center at the origin under the
mapping $z\mapsto1/z$.


Let us recall the concept of a Schwarz function and some its generalizations.
Let $\varGamma$ be a simple closed analytic curve. It is well-known (see, for instance \cite[Sections 1,2]{Sha1992book}) that in
this case there exist an open set $U$, $\varGamma\subset U$, and a function
$S$ holomorphic in $U$, such that
$$
\varGamma=\{z\in U\colon \overline{z}=S(z)\}.
$$
The function $S$ is called a \emph{Schwarz function} of $\varGamma$.

Let now $G$ be a bounded (not necessarily simply connected) domain possessing
the following property: there exist a compact set $K\subset G$ and a function
$S$ holomorphic in $G\setminus K$, continuous up to $\partial G$, and such
that $\overline{z}=S(z)$ on $\partial G$. In the latter case the aforesaid
function $S$ is called the \emph{one-sided Schwarz function} of $\partial G$.
Let us mention here Theorem~5.2 in \cite{Sak1991acta} which
says that if the boundary of some domain admits the one-sided Schwarz function,
then it consists of finitely many analytic curves.

It is known that the boundary of any quadrature domain (even of any
quadrature domain in the wide sense) admits the one-sided Schwarz function, see
\cite[Section~4.2]{Sha1992book}. We recall that a quadrature domain in the
wide sense is a domain satisfying the following property: there exists a
distribution $T$ with support $\mathop{\mathrm{Supp}}(T)\subset G$ such that
for every holomorphic and integrable function $h$ in $G$ we have
$\iint_Gh(z)\,dxdy=T(f)$. If $T$ has finite support, then $G$ is a quadrature
domain (in the standard, or classical sense).

The readers interested in the concept of the Schwarz function, its generalizations
and applications to the theory of quadrature domains should turn to the books \cite{Dav1974book} and \cite{Sha1992book}, and to the
Harold S.~Shapiro anniversary volume \cite{OTAA156} editied by P.~Ebenfelt, B.~Gustafsson,
D.~Khavinson, and M.~Putinar. The survey paper by B.~Gustaffson and H.~S.~Sha\-piro opening this
volume and the references therein are especially useful.

The property of being a Nevanlinna domain is weaker than that of admitting the one-sided
Schwarz function. It is natural to compare the corresponding classes of domains.
It turns out that they are quite different.


\begin{theorem}\label{t0-1}
For every $\beta\in[1,2]$ there exists a domain $G\in\ND$ such that
$\dim_H(\partial G)=\beta$, where $\dim_H$ stands for the Hausdorff dimension
of sets.
\end{theorem}

This theorem is an immediate corollary of our main results (see
Theorems~\ref{t1} and~\ref{t2} below). 
Thus, we can get far away from domains
with piecewise analytic boundaries (and, therefore, from quadrature domains)
if we consider Nevanlinna domains instead of domains whose boundaries admit the
one-sided Schwarz function.

Constructing Nevanlinna domains with irregular (for
instance non-analytic, non-smooth, and even more irregular) boundaries is a
rather difficult and delicate problem. It was considered in
\cite{Maz1997mn,Fed2006psim,BarFed2011msb,Maz2016spmj,Maz2018spmj}. The
detailed account of these results will be given in Section~\ref{section2} below. We
highlight here only several results not paying attention to the details of
respective constructions.

The first example of $\mathcal N$-domain with nowhere analytic (but rather
smooth) boundary was constructed in \cite{Maz1997mn}. Later on, several
constructions of $\mathcal N$-domains with boundaries belonging to the class
$C^1$, but not to the class $C^{1,\alpha}$, $\alpha\in(0,1)$, were obtained
in \cite{Fed2006psim} and \cite{BarFed2011msb}. Furthermore, it was shown in
\cite{BarFed2011msb} that Nevanlinna domains may have ``almost'' non-rectifiable
boundaries.
The first example of an $\mathcal N$-domain with non-rectifiable boundary was
constructed in \cite{Maz2016spmj}. Finally, we mention the recent paper
\cite{Maz2018spmj}, where an example of Nevanlinna domain $G$ such that
$\dim_H(\partial G)>1$ was produced.

%

\subsection{Nevanlinna domains with analytic boundaries and univalent rational
functions}

Let $G$ be a Jordan domain with analytic boundary. We have already mentioned that
in this case there exist an open set $U$, $\partial G\subset U$, and a
holomorphic function $S$ in $U$ such that $\overline{z}=S(z)$ on $\partial
G$. In view of the Luzin--Privalov boundary uniqueness theorem, the domain $G$ in
this case is a Nevanlinna domain if and only if $S$ extends to a
meromorphic function in $G$.

It follows from \cite[Chapter~14, p.158]{Dav1974book} that $S$ is meromorphic in $G$
if and only if $G$ is the image of the unit disc under conformal mapping by
some rational function $R$ without poles on $\overline{\mathbb D}$ and univalent in $\mathbb D$.
Therefore, it is of interest to consider a quantitative version of the
problem on the existence of Nevanlinna domains with non-rectifiable boundaries.
Namely, one
studies the question on how the length of the boundary of the (Nevanlinna)
domain $R(\mathbb D)$ grows in relation to the degree of
the rational function $R$.

Given a positive integer $n$, let us denote by $\mathcal R_n$ the set of all
rational functions of degree at most $n$ (thus, $\mathcal R_n$ consists of
all functions of the form $P(z)/Q(z)$, where $P$ and $Q$ are polynomials of degree at most $n$) and by $\mathcal R\mathcal U_n$
the set of all functions from $\mathcal R_n$ without poles in
$\overline{\mathbb D}$ and univalent in $\mathbb D$. Finally, let
$\mathcal R\mathcal U_{n,1}$ be the set of all functions $R\in\mathcal
R\mathcal U_n$ such that $\|R\|_{\infty,\mathbb T}\le1$. Set
$$
\gamma_0=\limsup_{n\to\infty}\sup_{R\in\mathcal R\mathcal U_{n,1}}
\dfrac{\log\ell(R)}{\log{n}},\qquad\text{where}\quad
\ell(R):=\dfrac1{2\pi}\int_{\mathbb T}|R'(\zeta)|\,|d\zeta|.
$$
It is shown in \cite{BarFed2017jam} that
$0<B_b(1)\le\gamma_0\le1/2$, where $B_b(1)$ is the value at $1$ of
the so called boundary means spectrum for bounded univalent functions (see
\cite[Chapter~8]{Pom1992book} and \cite[Chapter~VIII]{GarMar2005book}). This inequality
means that the length of the boundary of the domain $R(\mathbb D)$,
$R\in\mathcal R\mathcal U_n$ can grow at least like $n^\gamma$ as
$n\to\infty$ for some $\gamma>0$.
%

It is easily seen that the Nevanlinna domains of the form $R(\mathbb D)$, $R\in\cup_{n\ge 1}\mathcal R\mathcal U_n$,
are dense in the set of all Jordan domains in
$\mathbb C$ in the Hausdorff metric. This fact together with the
observation concerning possible growth of the length of boundaries of such
domains makes more clear the fact why Nevanlinna domains with non-rectifiable
boundaries do exist.

Our next result is as follows.

\begin{theorem}\label{t0-2}
For some absolute constant $\alpha>0$ and for every $n\ge 1$ we have
$$
\alpha\sqrt{n} \le\sup_{R\in \mathcal R\mathcal U_{n,1}}\ell(R)\le6\pi\sqrt{n},
$$
so that $\gamma_0=1/2$.
\end{theorem}

The new result here is the lower estimate which we obtain by constructing a snake like domain $R(\mathbb D)$ with long boundary.
The upper estimate comes from \cite[Theorem 1.2, Proposition 1.3]{BarFed2017jam}.

Since $0.23<B_b(1)\le 0.46$ (see \cite{BelSmi2005ecm} and
\cite{HedShi2005duke} respectively), it follows from Theorem~\ref{t0-2} that
the value $B_b(t)$, $t>0$, of the boundary means spectrum for bounded
univalent functions cannot be attained at the class of univalent rational
functions. It is worth to note here that the boundary means spectrum for
univalent (not necessarily bounded) functions $B(t)$ in the case $t>0$ is
attained on a certain class of univalent polynomials, see \cite{Kay2008mn}.

\subsection{Organization of the paper}
In Section~\ref{section2} we briefly discuss the
problem of uniform approximation of function by polyanalytic polynomials and
show the role that the concept of a Nevanlinna domain plays in this problem.
Furthermore we present several properties and examples of Nevanlinna domains.
The main results of this paper are formulated in Section~\ref{section2b} in a complete and
detailed form. In Sections~\ref{section3} and \ref{section4} we give two different constructions of
Nevanlinna domains with large boundaries and prove Theorems~\ref{t1}
and~\ref{t2} correspondingly.
Furthermore, in Section~\ref{section3} we establish Theorem~\ref{t1-0}. This result is somewhat
weaker than Theorems~\ref{t1} and~\ref{t2}. On the other hand, its proof is much simpler than their proofs.
In Section~\ref{section5} we prove Theorem~\ref{t0-2}. Note
that the proofs of Theorems~\ref{t0-2} and~\ref{t2} are based on similar
elementary topology constructions.

\section{Background information on the Nevanlinna domains}
\label{section2}

\subsection{Nevanlinna domains in problems of polyanalytic polynomial approximation}
The concept of a Nevanlinna domain is closely related to uniform approximation of functions by
polyanalytic polynomials on compact sets in $\mathbb C$.

Let us recall that a function $g$ is called \emph{polyanalytic} of order
$n$ (for integer $n\ge 1$) or, in short, \emph{$n$-analytic}, on an open set
$U\subset\mathbb C$ if it is of the form
\begin{equation}\label{paf}
g(z)=g_0(z)+\ov{z}g_1(z)+\cdots+\ov{z}^{n-1}g_{n-1}(z),
\end{equation}
where $g_0,\ldots,g_{n-1}$ are holomorphic functions in $U$. Note that the
space of all $n$-analytic functions in $U$ consists of all continuous
functions $f$ on $U$ such that $\overline\partial^nf=0$ in $U$ in the sense of the
distributions, where $\overline\partial$ is the standard
Cauchy--Riemann operator. 
By $n$-analytic polynomials and
$n$-analytic rational functions we mean the functions of the form
\eqref{paf}, where $g_0,\ldots,g_{n-1}$ are polynomials and rational
functions in the complex variable respectively. Traditionally, $2$-analytic
functions are called \emph{bianalytic}.

The problems we are interested in is to describe the compact sets $X$ such that
every function $f$ continuous on $X$ and $n$-analytic on its interior
can be approximated uniformly on $X$ by $n$-analytic rational functions with
no singularities in $X$, or by $n$-analytic polynomials.

These problems have attracted attention of analysts since the beginning of
1980s, but the main efforts were focused on the problem of approximation by
polyanalytic rational functions (see, for instance, \cite{TreWan1981pams,
Car1985jat, Ver1993pjm} and a recent survey paper \cite{FedMazPar2012rms} for
a detailed account of this problem). J.~Verdera \cite{Ver1993pjm} formulated the
following conjecture: \emph{if $X$ is an arbitrary compact subset of the complex plane and if
$f$ is continuous on $X$ and bianalytic on its interior, then $f$ can be
approximated uniformly on $X$ by bianalytic rational functions without
singularities in $X$}.
Omitting here the reasons supporting this conjecture (an appropriate
discussion can be found in \cite{Ver1993pjm}) let us mention that it was
proved recently by M.~Mazalov \cite{Maz2004msb}. Later on this result was
generalized to the solutions of general elliptic equations with constant complex
coefficients and locally bounded fundamental solutions in \cite{Maz2008msb}
(see also \cite{Fed2018cvee} for yet more observations concerning the
matter).

At the same time, there was no substantial progress in the problem of
polyanalytic polynomial approximation until the middle of 1990s. In
\cite{Fed1996mn} the third author found a necessary and sufficient condition
on a rectifiable simple closed curve $\varGamma$ in order that the system of
$n$-analytic polynomials (for every integer $n\ge2$) is dense in the
space of continuous functions on $\varGamma$. In this result the concept of a
Nevanlinna domain (formulated in a slightly different way) has appeared in
the first time.

Later on, several interesting and important results about uniform
approximation by polyanalytic polynomials were obtained in \cite{Fed1996mn,
CarFedPar2002msb, BoiGauPar2004izv, CarFed2005otaa, BarCarFed2016jat}. The
keynote ingredient of these results is the concept of a Nevanlinna domain and
several its special refinements and modifications.

For instance, a criterion for the uniform approximation of functions by
polyanalytic polynomials on Carath\'eodory compact sets in $\mathbb C$ was
obtained in terms of Nevanlinna domains in \cite{CarFedPar2002msb}. Recall that a compact set $X$ is called a Carath\'eodory compact set if
$\partial X=\partial\widehat{X}$, where $\widehat{X}$ denotes the union of
$X$ and all bounded connected components of $\mathbb C\setminus X$.

\begin{proposition}[\rm see~\cite{CarFedPar2002msb}, Theorem~2.2]
Let $X$ be a Carath\'eodory compact set in $\mathbb C$, and $n\ge2$ be
an integer. In order that each function $f$ which is continuous on $X$ and
$n$-analytic inside $X$ can be approximated uniformly on $X$ by
$n$-analytic polynomials it is necessary and sufficient that every bounded
connected component of the set $\mathbb C\setminus X$ is not a Nevanlinna
domain.
\end{proposition}

Note that the approximation condition in this Proposition does not depend
on $n$. However, for more complicated compact sets, the approximation conditions do
depend on $n$, see \cite{CarFed2008mn}.

Let us also mention that Nevanlinna domains have arisen in problems of uniform approximation
of functions by polynomial solutions of general homogeneous second order
elliptic equations on planar compact sets (see, for instance,
\cite[Theorem~3]{Zai2004izv}).

\subsection{Two equivalent description of $\mathcal N$-domains}

The following characterization of Nevanlinna
domains turns out to be both interesting and useful.

\begin{proposition}[\rm see~\cite{CarFedPar2002msb}, Proposition
3.1]\label{pscont} A domain $G$ is a~Nevanlinna domain if and only if
a~conformal mapping $f$ of the unit disc $\mathbb D$ onto~$G$ admits
a~Nevanlinna-\allowbreak type pseudocontinuation, so that there exist two
functions $f_1,f_2\in H^\infty(\widehat{\mathbb C}\setminus\nobreak\overline{\mathbb D})$ such that $f_2\not\equiv0$ and for
almost all points $\zeta\in\mathbb T$ the equality
$f(\zeta)=f_1(\zeta)/f_2(\zeta)$ holds, where $f_1(\zeta)$ and $f_2(\zeta)$
are the angular boundary values of the functions $f_1$ and~$f_2$.
\end{proposition}

We mention several simple consequences of this descripton. If $G$~is
a Nevanlinna domain and $g$ is a rational function with poles
outside $\overline{G}$ which is univalent in $G$, then the domain $g(G)$ is also
a Nevanlinna domain. Moreover, Nevanlinna domains have the following
``density'' property: any neighbourhood of an arbitrary simple close curve
contains an analytic Nevanlinna contour (i.e. the boundary of some Jordan
Nevanlinna domain). In order to establish the latter property one needs to
take some conformal mapping from the unit disc onto the interior of the
contour under consideration (in view of Carath\'eodory extension theorem this
function is continuous in the closed unit disc), and to approximate it
uniformly on $\overline{\mathbb D}$ with appropriate rate by univalent
polynomials.

Next we need to establish some relations between the concept of a
Nevanlinna domain and the theory of model (sub)spaces. Recall that a function
$\varTheta\in H^\infty=H^{\infty}(\mathbb D)$ is called an \emph{inner
function} if $|\varTheta(\zeta)|=1$ for almost all $\zeta\in\mathbb T$.
Let us denote by $H^2$ the standard Hardy space.
For an inner function $\varTheta$ we define the space
$$
K_\varTheta:=(\varTheta H^2)^\perp=H^2\ominus\varTheta H^2.
$$
In view of the Beurling theorem, the spaces $K_\varTheta\subset H^2$
are exactly the invariant subspaces of the backward
shift operator $f\mapsto(f(z)-f(0))/z$ in~$H^2$. The spaces $K_\varTheta$ are
usually called \emph{model spaces} (or \emph{model subspaces}): this
terminology was suggested by N.~Nikolski in view of the remarkable
role these spaces play in the functional model of Sz.-Nagy and Foia\c s.

\begin{proposition} {\rm (see \cite[Theorem~1]{Fed2006psim}, \cite[Theorems~A and~B]{BarFed2011msb})} Let $G$ be a bounded simply connected domain in $\mathbb C$
and let $f$ be some conformal mapping from $\mathbb D$ onto $G$. If
$G\in\mathcal N$, then there exists an inner function $\varTheta$ such that
$f\in K_\varTheta$. Reciprocally, if $\varTheta$ is an inner function, then
any bounded univalent function from the space $K_\varTheta$ maps $\mathbb D$
conformally onto some Nevanlinna domain.
\end{proposition}

\subsection{Univalent functions in $K_\varTheta$ and constructions of
Nevanlinna domains}

The above proposition gives us the following method for constructing
Nevanlinna domains: in the space~$K_\varTheta$ (for a special choice of
inner function $\varTheta$) one finds a univalent function which
possesses certain analytic properties (for instance, a function having some
prescribed boundary behaviour).

The description of inner functions $\varTheta$ for which the corresponding
space $K_\varTheta$ contains bounded univalent functions was recently obtained
in \cite{BelFed2018rms}. Recall that every inner function $\varTheta$ can be
expressed in the form $\varTheta(z)=e^{ic}B(z)S(z)$, where $c$ is some
positive constant, while $B$ and $S$~are some Blaschke product and singular
inner function respectively. Let us also recall that a~Blaschke product is
a~function of the form
\begin{equation}
\label{eq1.3}
B(z)=\prod_{n=1}^\infty\frac{\overline{a}_n}{|a_n|}\,
\frac{z-a_n}{\overline{a}_nz-1},
\end{equation} where
$(a_n)_{n=1}^\infty$ is some Blaschke sequence in $\mathbb D$ (that
is,~$a_n\in\mathbb D$ for $n\in\mathbb N$ and
$\sum_{n=1}^\infty(1-|a_n|)<\infty$), while a~singular inner function is
a~function of the form
\begin{equation}
\label{eq1.4}
S(z)
=\exp\biggl(-\int_{\mathbb T}\frac{\zeta+z}{\zeta-z}\,d\mu_S(\zeta)\biggr),
\end{equation} where
$\mu_S$~is some finite positive singular (with respect to the arc length) measure on $\mathbb T$. The result
established in \cite[Theorem 1]{BelFed2018rms} is as follows.

\begin{proposition}\label{prop-univms}
Let $\varTheta$ be an inner function in $\mathbb D$. The space $K_\varTheta$
contains bounded univalent functions if and only if one of the following two
conditions is satisfied:

\smallskip
{\rm (i)} $\varTheta$ has a~zero in $\mathbb D$;

\smallskip
{\rm (ii)} $\varTheta=S$ is a singular inner function and the measure $\mu_S$
is such that $\mu_S(E)>0$ for some Beurling--Carleson set $E\subset\mathbb T$, which
means that $\displaystyle\int_{\mathbb
T}\log\operatorname{dist}(\zeta,E)\,|d\zeta|>-\infty$.
\end{proposition}

Beurling--Carleson sets first appeared as boundary zero sets of analytic functions in the disc
which are smooth up to the boundary. It is worth to mention that property
(ii) in the latter proposition is also a necessary and sufficient condition for the space $K_S$ to contain mildly smooth functions (e.g., from the
standard Dirichlet space in $\mathbb D$), see \cite{DyaKha2006cm}.

Let us return to the problem on how ``bad'' could be the boundary of
a Nevanlinna domain. In many situations, questions about the regularity or
irregularity of boundaries of planar domains may be reduced to the
corresponding questions about the boundary regularity of conformal mappings
of the disc $\mathbb D$ onto the domains under consideration. Thus, we need
to be able to find bounded univalent functions possessing certain boundary
regularity (or irregularity) properties in the spaces~$K_\varTheta$ for
specially chosen inner functions~$\varTheta$. It seems appropriate to study
this question separately in two distinct cases: (i) $\varTheta=B$ is a
Blaschke product, and (ii) $\varTheta=S$ is a singular inner function (it may
be readily verified that $K_{BS}=K_B\oplus BK_S$).

The first example of a Nevanlinna domain with nowhere analytic boundary was
constructed in \cite{Maz1997mn}. The respective domain was constructed as the
conformal image of the unit disc under a map $f$ of the
form
\begin{equation}\label{eq-aaa}
f(z)=\sum_{n=1}^{\infty}\frac{c_n}{1-\overline{a}_nz},
\end{equation}
where $(a_n)_{n\ge 1}$ is some (infinite) Blaschke sequence satisfying the Carleson
condition
$$
\inf_{n\in\mathbb{N}}\prod_{\substack{k=1\\k\ne n}}^\infty
\biggl|\frac{a_n-a_k}{1-a_n\overline{a}_k}\biggr|>0,
$$
and $(c_n)_{n\ge 1}$ is an appropriately chosen sequence of coefficients. Such Blaschke sequences are called interpolating, and for any interpolating
Blaschke sequence $(a_n)_{n\ge 1}$ the sequence of functions
$$
\biggl\{\frac{\sqrt{1-|a_n|^2}}{1-\overline{a}_nz}\biggr\}
$$
forms a Riesz basis in the corresponding space $K_B$.

In \cite[Theorem~3]{Fed2006psim} it was shown that for every $\alpha\in(0,1)$
there exists a Nevanlinna domain with boundary in the class $C^1$ but not in
the class $C^{1,\alpha}$. The construction in
\cite{Fed2006psim} is rather complicated and technically involved. The main idea is to use an orthonormal
basis in the space $K_B$ (namely, the Malmquist--Walsh basis) instead of the Riesz
basis consisting of the corrseponding Cauchy kernels. Later on
%
%
it was proved in \cite[Theorem~2]{BarFed2011msb} that for every
$\alpha\in(0,1)$ and for every closed subset $E\subseteq\mathbb T$ there exists
an interpolating Blaschke sequence $(a_n)_{n\ge1}$ such that the set of
its limit points is equal to~$E$, and the space $K_B$, where $B$ is the
corresponding Blaschke product, contains a univalent function $f$ of the form
\eqref{eq-aaa} which maps $\mathbb D$ conformally onto a Nevanlinna domain
$f(\mathbb D)$ with boundary in the class $C^1$ but not in the class
$C^{1,\alpha}$.

Furthermore, in \cite{BarFed2011msb}, there is a construction of a function $f$ of the
form \eqref{eq-aaa} such that $f$ is univalent in $\mathbb D$ but $f'\not\in
H^p$ for any $p>1$. It means that the boundary of a Nevanlinna domain
$f(\mathbb D)$ is ``almost'' non-rectifiable. The first example of a Jordan
Nevanlinna domain with non-rectifiable boundary was constructed in
\cite{Maz2016spmj}. The corresponding domain is also $f(\mathbb D)$,
for some function $f$ of the form \eqref{eq-aaa}
univalent in the unit disc.

Finally, in \cite{Maz2018spmj} an example of a Nevanlinna domain $G$ such that
$\dim_H(\partial G)=\log_23$ was constructed. As before, $G=f(\mathbb D)$ for a suitable function $f$ of the
form \eqref{eq-aaa}.

Now let $S$ be a singular inner function. It follows from
Proposition~\ref{prop-univms} that if the measure $\mu_S$ has atoms, then
the space $K_S$ contains bounded univalent functions. In particular, this is
the case when $S(z)=\exp\bigl(\frac{z+1}{z-1}\bigr)$. Equivalently, the
Paley--Wiener space $\mathcal{P}W_{[0,1]}$, the Fourier image of $L^2[0,1]$,
considered as a space of functions analytic in the upper half-plane $\mathbb C_+$,
contains bounded univalent functions. Up to now only a few explicit examples
of bounded univalent functions in the Paley--Wiener space are known, and all
such examples map the upper half plane into domains with very regular
boundaries, see \cite{BelFed2018rms}.

\section{Main results}
\label{section2b}

First we introduce some standard notation.
We start with the concept of the Hausdorff dimension of sets. The definition may be found, for instance, in \cite[Chapter~4]{Mat1995book}, but we
present it here for the sake of completeness.
Let $D(a,r)$ stand for the open disc with center at the point $a\in\mathbb C$ and with radius $r>0$.
For a bounded set
$E\subset\mathbb C$ its $s$-dimensional Hausdorff measure $\mathcal H^s(E)$
is defined as follows:
$$
\mathcal H^s(E)=\lim_{\delta\to0}\inf_{\{D_j\}}\sum_{j}r_j^s,
$$
where the infimum is taken over all coverings of $E$ by families
of discs $\{D_j\}$, $D_j=D(z_j,r_j)$, of radius at most $\delta$ (it is clear that instead of the discs $D_j$ one can consider squares
of side length at most $\delta$). By definition, the Hausdorff dimension $\dim_H(E)$
is the unique number such that $\mathcal H^s(E)=\infty$ for every
$s<\dim_H(E)$, while $\mathcal H^t(E)=0$ for every $t>\dim_H(E)$.

Given a bounded simply connected domain $G$ we consider the set
$\partial_aG\subset\partial G$ which consists of all points of $\partial G$
being accessible from $G$ by some curve. According to \cite[Propositions~2.14
and~2.17]{Pom1992book}, the equality
$$
\partial_aG=\bigl\{f(\zeta)\colon \zeta\in\mathcal F(f)\bigr\}
$$
takes place, where $f$ is some conformal mapping from the unit disc $\mathbb
D$ onto $G$ and $\mathcal F(f)$ is its Fatou set, that is the set of all
points $\zeta\in\mathbb T$, where the angular boundary values $f(\zeta)$ exist.
It can be shown that $\partial_aG$ is a Borel set (see, for instance,
\cite[Section 2]{CarFed2005otaa}). It is clear that the set $\partial_aG$
depends only on the domain $G$ but not on the choice of $f$.

The definition of Nevanlinna domains (see \eqref{nd1} and its
interpretation \eqref{nd2}), imposes conditions
only on the accessible part $\partial_aG$ of their boundaries. By this
reason it seems more accurate and adequate to pose the question about
the existence of Nevanlinna domains with large \emph{accessible} boundaries.

\begin{theorem}\label{t1}
For every $\beta\in[1,2]$ there exists a function $f$ of the form \eqref{eq-aaa} univalent in $\mathbb D$ and such that
the Nevanlinna domain $G=f(\mathbb D)$ satisfies the property
$\dim_H(\partial_aG)=\beta$.
\end{theorem}

Note that the function $f$ from Theorem~\ref{t1} belongs to the space
$K_B$ for some appropriately chosen Blaschke product $B$. We would like to construct similar examples working with univalent function from the
space $K_S$, where $S$ is some singular inner function. The simplest example of such a space $K_S$ is the Paley--Wiener space
$\mathcal{P}W^{\infty}_{[0,1]}$ (which is considered, as mentioned
above, as the space of functions analytic in the upper half-plane $\mathbb C_+$).

\begin{theorem}\label{t2}
For every $\beta\in[1,2]$ there exists a univalent function $f$ belonging to the space
$\mathcal{P}W^{\infty}_{[0,1]}$ such that
the Nevanlinna domain $G=f(\mathbb C_+)$ satisfies the property
$\dim_H(\partial G)=\beta$.
\end{theorem}

\section{Proof of Theorem~\ref{t1} and related topics}
\label{section3}

Before proving Theorem~\ref{t1} we establish one more simple result of the same
nature. Namely, in Theorem~\ref{t1-0} below we give a hedgehog like construction of a Nevanlinna domain $G$
such that $m_2(\partial G)>0$.
To formulate this theorem we need yet another concept of dimension of sets.

The Hausdorff dimension is defined by considering all coverings of a given
set by small balls $D_j=D(z_j,r_j)$ and inspecting the sums
$\sum_jr_j^s$. One natural modification of this definition of
dimension is obtained when we consider coverings with balls (cubes) of the same
size. Such modification leads to the concept of the Minkowski dimension (or the
box-counting dimension) $\dim_M$, see \cite[Section~5.3]{Mat1995book} and
\cite[Section 3.1]{Fal2003book}. Skipping here the formal definition of Minkowski
dimension we recall that the value $\dim_M(E)$ of a bounded non-empty set
$E$ is calculated as
$$
\lim_{N\to\infty}\frac{\log M_E(N)}{N},
$$
where $M_E(N)$ is the minimal number of cubes (boxes) of side length $2^{-N}$
required to cover $E$.

It can be verified that
$$
\dim_H(E)\le\dim_M(E)\le2
$$
and both inequalities can be strict.

\begin{theorem}\label{t1-0}
There exists a function $f$ of the form
\eqref{eq-aaa} univalent in $\mathbb D$ such that the Nevanlinna domain $G=f(\mathbb D)$
satisfies the properties $m_2(\partial G)>0$, $\dim_M(\partial_aG)=2$.
\end{theorem}


\begin{proof}
We start with the following building block, sometimes called ``Mazalov's needle", see
\cite[Section~2]{Maz2016spmj}. For every sufficiently small $b>0$ there
exists a rational function $F_b$ with simple poles $\{w_k\}_{k=1}^L$ in
$\mathbb C\setminus \overline{\mathbb D}$ such that
\begin{align}
|F_b(z)|+|F'_b(z)|\le b,&\qquad z\in \mathbb D\setminus D(1,\sqrt{b}),\notag\\
|F_b(z)|\le b\ \text{and}\ |F'_b(z)|\le 1,&\qquad z\in \mathbb D\setminus D(1,b),\notag\\
\Re F'_b(z)\ge -\frac 12,&\qquad z\in \overline{\mathbb D},\notag\\
|\Im F_b(z)|\le b,&\qquad z\in \mathbb D\cap D(1,b),\notag\\
F_b(1)=3,&\\
\intertext{and, finally,}
\sum_{k=1}^L(|w_k|-1)\le b&.
\label{ft2}
\end{align}
\smallskip

For $I\subset\mathbb R_+$, $E\subset[0,2\pi)$ we use the notation
$$
S(I,E)=\Bigl\{re^{i\theta}\in\mathbb C\colon  r\in I,\, \theta\in E\Bigr\}.
$$

Let us choose a nowhere dense compact set $K$ of positive one-dimen\-sional Lebesque measure on the unit circle $\mathbb T$. We have $\mathbb T\setminus K=\bigsqcup_{j\ge 1} I_j$,
where $I_j=\{e^{it}\colon |t-\alpha_j|<\gamma_j\}$ are open arcs.
Set $I^*_j=\{e^{it}\colon |t-\alpha_j|<\gamma_j/2\}$.

We define a sequence $\{\varphi_n\}$ of rational functions, a sequence of unimodular numbers $\{e^{i\theta_n}\}$ and a sequence of positive numbers
$\{b_n\}$ in the
following inductive procedure. Set $\varphi_0(z)=z$. On the step $n\ge0$ we have
$$
\varphi_n(z)=z+\sum_{j=1}^n e^{i\theta_j}F_{b_j}(ze^{-i\theta_j}).
$$
We assume that $\varphi_n$ satisfies the following properties:
\begin{align*}
(a) \quad &\text{the set $\varGamma_n=\varphi_n(\mathbb T)$ is a simple closed curve},\\
(b) \quad &\arg\varphi_n(\overline{D(e^{i\theta_j},b_j)\cap \mathbb D})\subset I^*_j,\qquad 1\le j\le n,\\
(c) \quad &\varGamma_n\subset S\bigl((0.5,1.5),[0,2\pi)\bigr)\cup
\bigcup_{j=1}^n S\bigl((1,4),I^*_j\bigr),
\\
(d) \quad &|\varphi_n(e^{i\theta_j})|>2,\qquad 1\le j\le n,\\
(e) \quad &\text{the index of $\varGamma_n$ with respect to the point $0$ is equal to $1$.}
\end{align*}

By (e), for every $t\in[0,2\pi]$ there exists $\theta\in[0,2\pi]$ such that
$$
\arg\varphi_n(e^{i\theta})=t.
$$
Therefore, we can choose
$\theta_{n+1}\in[0,2\pi)$ satisfying
$$
\arg\varphi_n(e^{i\theta_{n+1}})=\alpha_{n+1}.
$$
Set
$$
\varphi_{n+1}(z)=\varphi_n(z)+ e^{i\theta_{n+1}}F_{b}(ze^{-i\theta_{n+1}}).
$$
For sufficiently small $b\in(0,2^{-n})$ the condition (b) on $\varphi_{n+1}$
holds for $1\le j\le n$ and for $j=n+1$ by continuity. The same
is true for (c) and (d). A simple continuity argument together with condition
(c) shows that the index of $\varGamma_{n+1}$ with respect to $0$ is equal to
$1$. Fix such small $b<2^{-n}$ and denote it by $b_{n+1}$. Since $\Re \varphi'_{n+1}\ge 1/4$ on $\mathbb D$, the function $\varphi_{n+1}$
is univalent. Since $\varphi_{n+1}$ is rational, we obtain (a).
This completes the induction step.

We define
$$
\varphi(z)=z+\sum_{j\ge1}e^{i\theta_j}F_{b_j}(ze^{-i\theta_j}).
$$
The function $\varphi$ is analytic in the unit disc and belongs to the space
$K_B$ for some Blaschke product $B$ (it follows from property \eqref{ft2} of
the function $F_b$ and the estimate $b_{n+1}<2^{-n}$). Since the arcs $I_j$ are disjoint, by property (b) the
sets $\overline{D(e^{i\theta_j},b_j)\cap \mathbb D}$ are disjoint. Now, the
estimates on the derivative of $F_b$ yield that $\varphi$ is univalent
(because $\Re\varphi'>0$ on $\mathbb D$).
Thus, $G=\varphi(\mathbb D)$ is a (hedgehog like) 1Nevanlinna domain.

Next, by (c),
$$
\varphi(\mathbb D)\subset \overline{D(0,1.5)}\cup
\bigcup_{j\ge 1} S\bigl((1.5,4],I^*_j\bigr).
$$
The function $\varphi$ is continuous on $\overline{D(e^{i\theta_j},b_j)\cap
\mathbb D}$, $j\ge 1$, and
$$
|\varphi(e^{i\theta_j})|\ge 2,\qquad j\ge 1.
$$
By continuity of $\varphi$ on $\mathbb D$, we have
$$
\varphi(\mathbb D)\cap S(\{r\},I^*_j)\not=\varnothing, \qquad
1.5<r\le 2,\ j\ge 1,
$$
and hence,
$$
\partial\varphi(\mathbb D)\cap S(\{r\},I^*_j)\not=\varnothing, \qquad
1.5\le r\le 2,\ j\ge 1,
$$
Therefore,
$$
\partial\varphi(\mathbb D)\supset S([1.5,2],K)
$$
and hence,
$$
m_2(\partial\varphi(\mathbb D))>0.
$$

In order to finish the proof of Theorem~\ref{t1-0} we need to calculate the
Minkowski dimension 
of the set $\partial_aG$ (i.e. the set of the accessible points of the boundary
of $G=\varphi(\mathbb D)$). 

Suppose that the set $K$ satisfies the condition $\gamma_n\gtrsim
\dfrac1{n\log^2(n+1)}$, $n\ge1$. For every $j\ge1$, to cover the
set
$$
\partial_aG\cap S\bigl([1.5,2],I^*_j\bigr)
$$
we need at least $2^{N-1}$ boxes of side length $2^{-N}$. Since for different
$n$ with $\gamma_n>2^{1-N}$ these boxes are disjoint, we obtain that
$$
M_{\partial_aG}(N)\gtrsim 2^N\,{\rm card\,}\{n:\gamma_n> 2^{1-N}\},
$$
which yileds that $\dim_M(\partial_aG)=2$.
\end{proof}

\begin{remark}
N.~Makarov proved that for every simply connected domain, the support of
harmonic measure has Hausdorff dimension $1$. Later on, P.~Jones and T.~Wolff
proved that for every planar domain, the support of harmonic measure has
Hausdorff dimension at most $1$. For these results see
\cite[Section~6.5]{GarMar2005book}. In order to link this observation with
our subject we need to recall that the harmonic measure on $\partial G$
lives on $\partial_aG$, which means that the harmonic measure of the set
$E\setminus\partial_aG$ is zero for any Borel set $E\subset\partial G$. Moreover,
in the definition of a Nevanlinna domain we are dealing with the equality
\eqref{nd1} which holds, essentially, on $\partial_aG$.
\end{remark}

\begin{proof}[Proof of Theorem~\ref{t1}]
Fix $\varepsilon\in(0,1)$. We are going to construct a Nevanlinna domain
$G=G_{\varepsilon}$ such that
$\dim_H(\partial_aG)=2-\varepsilon$. In order to construct a Nevanlinna
domain $G$ with $\dim_H(\partial_aG)=2$ we just need to merge our
constructions with $\varepsilon_k\to0$, $k\to\infty$, see Step VI below.

\subsubsection*{Step~I.~ Binary words}

Denote by $\mathcal W$ the set of all binary words, i.e. words in the
alphabet $\{0,1\}$. For a given word $\omega\in\mathcal W$ we denote by
$|\omega|$ its length (i.e. the number of digits in $\omega$) and by
$\sum\omega$ the sum of its digits. Furthermore, we set
$\sign\omega=|\omega|-\sum\omega$. Given two words
$\omega_1,\omega_2\in\mathcal W$ we denote by
$\omega_1\omega_2=\omega_1{\cdot\,}\omega_2$ their concatenation. The empty word
will be denoted by $\mathfrak e$. Finally, for a word $\omega=\alpha\beta$,
where $\alpha,\beta\in\mathcal W$ and $|\beta|=1$ we put
$\widetilde\omega:=\alpha$.

\subsubsection*{Step~II.~$H$-tree and its neighborhood $\Omega$}

Fix $\varepsilon\in(0,10^{-2})$ and set $\lambda=2^{-1/2}-\varepsilon$. We
define a system of (closed) intervals
$I_\omega:=I_{z_\omega,\zeta_\omega}:=[z_\omega,z_\omega+\zeta_\omega]$,
$\omega\in\mathcal W$. Set $\psi_{\mathfrak e}(u)=u$,
$I_{\mathfrak e}=I_{0,1}=[0,1]$, so that
$z_{\mathfrak e}=0$ and $\zeta_{\mathfrak e}=1$. Furthermore, we define the
mappings
\begin{align*}
\psi_{\omega\cdot1}\colon u\mapsto z_\omega+(1-\varepsilon)\zeta_\omega+i\lambda\zeta_\omega u,\\
\psi_{\omega\cdot0}\colon u\mapsto z_\omega+(1-\varepsilon)\zeta_\omega-i\lambda\zeta_\omega u,\\
\end{align*}
and the segments $I_{\omega\cdot0}$ and $I_{\omega\cdot1}$ of the next
generation
\begin{align*}
I_{\omega\cdot1}:=\psi_{\omega\cdot1}(I_{\mathfrak e})=I_{z_\omega+(1-\varepsilon)\zeta_\omega,i\lambda\zeta_\omega},\\
I_{\omega\cdot0}:=\psi_{\omega\cdot0}(I_{\mathfrak e})=I_{z_\omega+(1-\varepsilon)\zeta_\omega,-i\lambda\zeta_\omega}.
\end{align*}
We have
\begin{gather*}
I_{\omega{\cdot}1},I_{\omega{\cdot}0}\perp I_{\omega},\\
|I_{\omega{\cdot}1}|=|I_{\omega{\cdot}0}|=\lambda| I_{\omega}|.
\end{gather*}
Furthermore, if $\omega_1,\omega_2\in\mathcal W$, then
\begin{equation}
\psi_{\omega_2}\circ \psi_{\omega_1}=\psi_{\omega_1{\cdot}\omega_2}.
\label{e1}
\end{equation}

Set
\begin{align*}
\mathcal H&=\bigcup_{\omega\in\mathcal W}I_{\omega},\\
\mathcal H_\infty&=\overline{\mathcal H}\setminus \mathcal H.
\end{align*}

Let now $\Omega_{\mathfrak e}$ be the $\varepsilon/100$-neighborhood of
$I_{\mathfrak e}$,
\begin{align*}
\Omega_\omega&=\psi_{\omega}(\Omega_{\mathfrak e}),\\
\Omega&=\bigcup_{\omega\in\mathcal W}\Omega_\omega.
\end{align*}

Next we establish several geometrical properties of the above described fractal
construction.

\begin{lemma}\label{le1}
\item{\rm(a)} Every point of $\mathcal H_\infty$ is an accessible point of $\partial\Omega$.
\item{\rm(b)} $\dim_H(\mathcal H_\infty)\ge 2-10\varepsilon$.
\item{\rm(c)} If $\omega\in\mathcal W$, then
$$
\diam(\Omega_\omega)\asymp \lambda^{|\omega|}.
$$
\item{\rm(d)} If $\omega_1,\omega_2\in\mathcal W$ and $\omega_1\ne\omega_2{\cdot}s$,
$\omega_2\ne\omega_1{\cdot}s$, $s\in\{0,1\}$, then
$$
\dist(\Omega_{\omega_1},\Omega_{\omega_2})\gtrsim \lambda^{\min(|\omega_1|,|\omega_2|)}.
$$
\end{lemma}

\begin{proof}
Properties (c) and (d) are easily verified for $\omega_1=\mathfrak e$. After
that, we just apply the self-similarity property \eqref{e1}.

Next, property (a) follows immediately from the construction of $\Omega$.

Finally, property (b) is a direct consequence of Frostman's lemma (see, for
example, \cite[Section~8]{Mat1995book}). It suffices to consider the weak
limit of the probability measures equidistributed (with respect to the
length) on $\bigcup_{\omega\in\mathcal W\colon |\omega|=n}I_{\omega}$,
$n\to\infty$.
\end{proof}

\subsubsection*{Step~III.~Mazalov type construction}

Our next ingredient is a Mazalov type lemma, compare to \cite{Maz2016spmj}.

Set $\mathbb C_L=\{z\in\mathbb C\colon \Re{z}\le 0\}$.

\begin{lemma}\label{le2}
Given $b\in(0,10^{-2})$, there exists a rational function
$$
F(z)=F_b(z)=\sum_{k=1}^M\frac{c_k}{z-w_k}
$$
with $c_k>0$, $w_k>0$, $1\le k\le M$, such that
\begin{align*}
(a) \quad &\text{$|F(z)|+|F'(z)|\le Ab$, $z\in\mathbb C_L\setminus D(0,b)$},\\
(b) \quad &\text{$\Re F(z)\ge -Ab^2$, $\Re F'(z)\ge -Ab$, $z\in\mathbb C_L$,}\\
(c) \quad &\text{$|\Im F(z)|\le Ab$, $z\in\mathbb C_L$},\\
(d) \quad &\text{$|F(0)-1|\le Ab^2$, $\Re F(z)\le 1+Ab^2$, $z\in \mathbb C_L$},\\
(e) \quad &\text{$\sum_{k=1}^M(c_k+w_k)\le Ab$},\\
(f) \quad &\text{If $t\in[1,3]$, $\delta=\exp(-2(1-t\varepsilon)/b^2)$, $\gamma\in(\pi/2,3\pi/2)$, then}\\
&\qquad\qquad
|\Re F(\delta e^{i\gamma})-(1-t\varepsilon)|\le Ab^2,\\
&\qquad\qquad\Bigl| \frac{b^2F'(\delta e^{i\gamma})}{2\exp(2(1-t\varepsilon)/b^2)}+e^{-i\gamma}\Bigr|\le Ab,
\end{align*}
for some absolute constant $A>0$.
\end{lemma}

Thus, the image of the left half-plane under the map $z\to z+F(z)$ is the union of the slightly perturbed left half-plane and a thin domain (needle)
close to the interval $[0,1]$. Property (f) means that we have good control on $F'(z)$ while $F(z)$ is close to $1-\varepsilon$ and $\Re z$ is close to $0$.

\begin{proof}
Let $N$ be the integer part of $\exp(b^{-2})$. We start with the function
$$
G(z)=\frac{b^2}{2}\int_{bN^{-2}}^b \frac{dt}{t-z}=\frac{b^2}2\log\frac{b-z}{bN^{-2}-z},\qquad z\in\mathbb C_L,
$$
where $\log$ is the principal branch of the logarithm function.

This function has the following simple properties:
\begin{gather}
G(0)=\max_{\mathbb C_L}\Re G=b^2\log N,\label{ft3}\\
\Re G(z)> 0,\qquad z\in\mathbb C_L,\label{ft4}
\end{gather}
and
\begin{equation}
|G(z)|\le \frac{Ab^3}{|z|},\quad |G'(z)|\le \frac{Ab^3}{|z|^2}, \qquad z\in\mathbb C_L\setminus D(0,b),
\label{e2}
\end{equation}
for some absolute constant $A>0$. Furthermore,
\begin{gather}
\Re G'(z)\ge -b,\qquad z\in\mathbb C_L,\label{e3}\\
|\Im G(z)|\le \pi b^2, \qquad z\in \mathbb C_L.\label{e4}
\end{gather}
Finally, if $t\in[1,3]$, $\delta=\exp(-2(1-t\varepsilon)/b^2)$, $\gamma\in(\pi/2,3\pi/2)$, then
\begin{gather}
|\Re G(\delta e^{i\gamma})-(1-t\varepsilon)|\le Ab^2,\label{ft5}\\
\Bigl| \frac{b^2G'(\delta e^{i\gamma})}{2\exp(2(1-t\varepsilon)/b^2)}+e^{-i\gamma}\Bigr|\le Ab.\label{ft6}
\end{gather}


Next, like in \cite{Maz2016spmj}, we use the Newton--Cotes quadrature formula of degree $2$ (the Simpson quadrature formula).
This formula claims that given an interval $[\alpha,\beta]\subset\mathbb R$ and
$f\in C^4([\alpha,\beta])$, we have
\begin{equation}
\Bigl|\int_\alpha^\beta f(x)\,dx -\frac{\beta-\alpha}{6}Q\Bigr|\le
\frac{(\beta-\alpha)^5}{2880}\max_{x\in [\alpha,\beta]}|f^{(4)}(x)|,
\label{e5}
\end{equation}
where
$$
Q=\sum_{j=0}^2d_jf\Bigl(\frac{j\alpha+(2-j)\beta}{2}\Bigr),
\quad d_0=d_2=1,\ d_1=4.
$$

Now we split the interval $[bN^{-2},b]$ into $N-1$ subintervals
$[bk^{-2},b(k-1)^{-2}]$, $2\le k\le N$, and set
$$
F(z)=\frac{b^2}{12}\sum_{k=2}^N\Bigl(\frac{b}{(k-1)^2}-\frac{b}{k^2}\Bigr)
\sum_{j=0}^2\frac{d_j}{\frac{jb}{2k^2}+\frac{(2-j)b}{2(k-1)^2}-z}.
$$
Then $F$ is a finite sum of simple fractions $c_k/(z-w_k)$ with
$w_k\in[bN^{-2},b]$, $c_k>0$,
$$
\sum_kw_k\le Ab,\qquad \sum_kc_k\le Ab^3,
$$
for some absolute constant $A$, and property (e) follows.

Applying estimate \eqref{e5} with $f(x)=1/(x-z)$ and with $f(x)=1/(x-z)^2$ we obtain
\begin{equation}
|G^{(j)}(z)-F^{(j)}(z)|\le
Ab^2\sum_{k\ge 1}\Bigl(\frac{b}{k^3}\Bigr)^5\Bigl(\frac{k^2}{b}\Bigr)^{5+j}
\le A_1b^{2-j},
\label{e6}
\end{equation}
for $z\in\mathbb C_L$, $j=0,1$, and for some absolute constants
$A$, $A_1$.

Now, \eqref{e6} and \eqref{ft3}--\eqref{e4} give properties (a)--(d).

Finally, property (f) follows from \eqref{e6}, \eqref{ft5}, \eqref{ft6}.
\end{proof}

\subsubsection*{Step~IV.~Conformal maps}

Consider an enumeration $\mathcal W=\{\omega_n\}_{n\ge0}$ such that if
$\omega_n,\omega_m\in\mathcal W$ and $|\omega_n|<|\omega_m|$, then $n<m$. In
particular, $\omega_0=\mathfrak e$. Denote $\mathcal
W_N=\{\omega_0,\omega_1,\ldots,\omega_{N-1}\}$, $\mathcal W_0=\varnothing$.

Set $\varphi_0(z)=z-1$. Then $\varphi_0(\mathbb D)\subset\mathbb C_L$. The
functions $\varphi_n$, $n\ge1$, will be constructed in the following inductive procedure.

On step $N\ge 0$ we have a set $\{b_\omega\colon \omega\in\mathcal
W_N\}$ of positive numbers, a set $\{\theta_\omega\colon \omega\in\mathcal
W_N\}$ of real numbers and a rational function
$$
\varphi_N(z)=(z-1)+
\sum_{\omega\in\mathcal W_N}(-1)^{\sign(\omega)}(i\lambda)^{|\omega|}
F_{b_\omega}(ze^{-i\theta_\omega}-1)
$$
such that
$$
\varphi_N(\mathbb D)\subset\mathbb C_L\cup \bigcup_{\omega\in\mathcal W_N}\Omega_\omega
$$
and for every $\omega\in\mathcal W_N$, $x\in I_\omega$,
$$
\dist(x,\varphi_N(\mathbb D))<\frac{\varepsilon\lambda^{|\omega|}}{100}.
$$

Given $\omega\in\mathcal W$ and $b>0$, set
\begin{gather*}
O_{\omega,b}:=\{z\in\mathbb D\colon\Re(ze^{-i\theta_\omega})>1-b\},\\
d_\omega:=\exp\Bigl(-\frac{3}{b^2_\omega}\Bigr),\\
U_\omega=\overline{O_{\omega,d_{\widetilde\omega}}}\setminus
(O_{\omega\cdot1,b_{\omega\cdot1}}\cup O_{\omega\cdot0,b_{\omega\cdot0}}).
\end{gather*}
We have
\begin{equation}
\Re\bigl(\varphi'_N(z)(-1)^{\sign(\omega)}i^{|\omega|}\bigr)>\frac 12,\qquad z\in U_\omega,\, \omega\in\mathcal W_N\setminus\{\mathfrak e\}.
\label{e7}
\end{equation}
If
$\omega{\cdot}1\not\in \mathcal W_N$, then we define
$U_\omega=\overline{O_{\omega,d_{\widetilde\omega}}}\setminus
O_{\omega\cdot0,b_{\omega\cdot0}}$, and make an analogous modification if
$\omega{\cdot}0\not\in \mathcal W_N$ or if $N=1$.

Furthermore, if $N\ge 2$, then
\begin{equation}
\Re(\varphi'_N(z))>\frac 12,\qquad z\in U_{\mathfrak e}=
\overline{\mathbb D}\setminus(O_{1,b_1}\cup O_{0,b_0}).
\label{e8}
\end{equation}
Thus, $\varphi_N$ is univalent on every set $U_\omega$, $\omega\in\mathcal
W_N$.

Next, if $\omega_1,\omega_2\in\mathcal W_N$, $\omega_1\not=\omega_2$,
$\omega_1\not=\omega_2^*$, $\omega_2\not=\omega_1^*$, $x_1\in U_{\omega_1}$,
$x_2\in U_{\omega_2}$, then
\begin{equation}
|\varphi_N(x_1)-|\varphi_N(x_2)|> A\lambda^{\min(|\omega_1|,|\omega_2|)},
\label{e9}
\end{equation}
for some absolute constant $A>0$.

If $\omega,\omega{\cdot}s\in\mathcal W^N$ for some $s\in\{0,1\}$,
$x_1\in \overline{U_{\omega}\setminus U_{\omega{\cdot}s}}$, $x_2\in \overline{O_{\omega{\cdot}s,b_{\omega{\cdot} s}}}$, then
\begin{equation}
|\varphi_N(x_1)-|\varphi_N(x_2)|> d_\omega.
\label{e10}
\end{equation}
As a consequence, $\varphi_N$ is univalent on $\mathbb D$.
The case $N=1$ is treated in a similar way.

Let $\omega_k=\widetilde\omega_N$. Without loss of generality assume that
$\omega_N=\omega_k\cdot1$. Set $I=I_{\omega_k}=[z,z+\zeta]$. Choose
$\theta_{\omega_N}>\theta_{\omega_k}$ such that the projection of
$\varphi_N(e^{i\theta_{\omega_N}})$ onto $I$ is $z+(1-\varepsilon)\zeta$, and set
$$
\varphi_{N+1}(z)=\varphi_N(z)+(-1)^{\sign(\omega_N)}(i\lambda)^{|\omega_N|}F_b(ze^{-i\theta_{\omega_N}}-1).
$$
Then by Lemma~\ref{le2}~(a), (c), and (d),
$$
\varphi_{N+1}(\mathbb D)\subset\mathbb C_L\cup \bigcup_{\omega\in\mathcal W^{N+1}}\Omega_\omega
$$
and for every $x\in I_\omega$, $\omega\in \mathcal W_{N+1}$,
$$
\dist(x,\varphi_{N+1}(\mathbb D))<\frac{\varepsilon\lambda^{|\omega|}}{100}
$$
for sufficiently small positive $b<2^{-n}$.

Furthermore, for sufficiently small $b$, inequalities \eqref{e7}, \eqref{e8}
hold for $\varphi_{N+1}$ and for $\omega\in \mathcal W_{N+1}$. Here we use
Lemma~\ref{le2}~(a) for $\omega\not=\omega_N$ and Lemma~\ref{le2}~(b),(f) for
$\omega=\omega_N$. Next, for sufficiently small $b$, inequalities \eqref{e9},
\eqref{e10} hold for $\varphi_{N+1}$ and for $\omega\in \mathcal W_{N+1}$.
Once again, we use Lemma~\ref{le1}~(d) for $\omega_1,\omega_2,\omega{\cdot}s\in\mathcal W_N$ and
Lemma~\ref{le2}~(a), (f), and (g) otherwise.
This completes the induction step.

\subsubsection*{Step~V.~Limit map.}

Passing to the limit $N\to\infty$ we obtain a univalent function $\varphi$ on
the unit disc such that
\begin{gather*}
\varphi(z)=(z-1)+\sum_{k\ge 1}\frac{c_k}{z-w_k},\\
\sum_{k\ge 1}\bigl(|c_k|+(|w_k|-1)\bigr)<\infty,\\
\varphi(\mathbb D)\subset\mathbb C_L\cup \Omega.
\end{gather*}
We have used here Lemma~\ref{le2}~(e).
Finally, for every $x\in\mathcal H_\infty$ there exists a path
$\gamma:[0,1)\to\mathbb D$ such that $x=\lim_{t\to 1}\varphi(\gamma(t))$, and
hence, $\mathcal H_\infty\subset\partial_a\varphi(\mathbb D)$.

\subsubsection*{Step~VI.~Dimension $2$.}

We choose a sequence of points on $\mathbb T$, say $\zeta_k=\exp(2^{-k}i)$ and $\varepsilon_k=2^{-10k}$. Set $\lambda_k=2^{-1/2}-\varepsilon_k$.
Next, we associate to every $\zeta_k$ a copy of $\mathcal W$ ordered as on Step IV, $\mathcal W^{(k)}=\{\omega^{(k)}_n\}_{n\ge 0}$.
Furthermore, we order the union of $\mathcal W^{(k)}$, $k\ge 1$, in a natural way: $\omega^{(1)}_0,\omega^{(1)}_1,\omega^{(2)}_0,\omega^{(1)}_2,\omega^{(2)}_1,\omega^{(3)}_0,\ldots$\,.
Now, using this ordering we construct the corresponding functions
$$
\varphi_{j,k}(z)=(z-1)+
\sum_{\omega^{(p)}_s\le \omega^{(j)}_k}(-1)^{\sign(\omega^{(p)}_s)}2^{-10p}(i\lambda_p)^{|\omega^{(p)}_s|}
F_{b_{\omega^{(p)}_s}}(ze^{-i\theta_{\omega^{(p)}_s}}-1).
$$
As on Step IV one verifies that $\varphi_{j,k}$ are conformal maps for suitable $\omega^{(p)}_s$ close to $\zeta_p$ and for sufficiently small $b_{\omega^{(p)}_s}$.
The limit univalent function satisfies the properties established on Step V and $\dim_H(\partial_aG)=2$.
\end{proof}

\section{Proof of Theorem~\ref{t2}}
\label{section4}

We use the construction of an $H$-tree described in the proof of
Theorem~\ref{t1} and some other notations from that proof. We suppose that $\dim_H(\mathcal H_\infty)$ is a fixed number in the interval $[1,2]$.

Applying a linear change of variables we can assume that $0\in\Omega\subset\mathbb D$.
A simple topological argument shows that there exists a $C^2$-smooth injective map $\gamma_0$
from the half-strip
$$
S=\{x+iy\in\mathbb C:x\ge0,\,|y|\le 1\}
$$
into $\mathbb D$ such that $\gamma_0=0$, $\gamma_0(S)\subset\Omega$ and $\gamma_0(S)\cap\Omega_\omega\not=\emptyset$ for every $\omega\in\mathcal W$. Changing, if necessary, the parametrization, we can assume that $\gamma=\gamma_0\,|\,\mathbb R_+$ satisfies the condition $|\gamma'(t)|=1$.

Then $\mathcal H_\infty\subset \overline{\gamma(\mathbb R_+)}$.

Choose a continuous function $\beta:\mathbb R_+\to(0,1)$ such that
$$
D(\gamma(t),100\beta(t))\subset\gamma_0(S), \qquad t\in[0,\infty).
$$
Set
$$
\delta(x)=\max_{[0,x]}\Bigl(\frac1{\beta}+|(\arg(\gamma'))'|\Bigr).
$$

Let $T>\max(100,\delta(1))$. Set $b_0=0$, $r_1=[2T\log T]+1$, $\rho=1-T^{-1}$. Next, for $n\ge 1$ we set
\begin{align*}
b_n&=b_{n-1}+\rho^{r_n},\\
r_{n+1}&=\begin{cases} r_n\text{\ \ if\ \ }\delta\bigl(b_n+T\rho^{r_n}\bigr)T\rho^{r_n}<1,\\
r_n+1 \text{\ \ otherwise}.
\end{cases}
\end{align*}

Then $b_n\nearrow\infty$ and $T\delta(b_n)(b_n-b_{n-1})<1$, $n\ge 1$.

Set $Q=T^{1/12}$, $\varepsilon=T^{-1/2}$, $w_n=\gamma(b_n)$, $a_n=w_{n+1}-w_n$, $n\ge 0$.
Then
\begin{enumerate}
\item[\rm(a)] $w_0=0$, $|w_n|<1$,
\item[\rm(b)] $1-\varepsilon<|a_{n+1}|\big/|a_n|<1+\varepsilon$,
\item[\rm(c)] $|\arg{a_{n+1}\overline{a_n}}|\le\varepsilon$,
\item[\rm(d)] $|a_n|\ge 2^{-n}T^{-2}$, $n\ge 0$.
\end{enumerate}

Define
$$
Q^{\pm}_n:=\conv\{w_n,w_{n+1}, w_n\pm 2ia_n, w_{n+1}\pm 2ia_{n}\},
$$
where $\conv\{A\}$ stands for the convex hull of $A$, and
$$
T^{\pm}_n=\conv\{w_{n+1}, w_{n+1}\pm2ia_{n}, w_{n+1}\pm2ia_{n+1}\}.
$$
Then $Q^{\pm}_n\cap Q^{\pm}_m=\varnothing$, and
$Q^{\pm}_n\cap T^{\pm}_m=\varnothing$ for $|n-m|>1$.
Furthermore, for sufficiently large $T$,
$$
\bigcup_n Q^{\pm}_n\cup T_n^{\pm}\subset \Omega.
$$

\begin{proposition}\label{mainprop}
There exists a meromorphic function $f$ with poles on the imaginary axis
which is univalent in the upper half plane $\mathbb C_+$ and such that $f(\mathbb
C_+)\subset\Omega$, $f(\mathbb R)\cap \gamma_0(x+i[-1,1])\not=\emptyset$, $x\ge x_0$.
\end{proposition}

\begin{proof}
Set $y_n=Q^n$, $y_{-n}=-Q^n$, $n\ge 1$,
$$
H(z)=\sum_{n=1}^\infty\frac{a_n i y_n}{(z+iy_n)^2},
\qquad f(z)=\int_0^zH(\zeta)d\zeta.
$$

\begin{lemma}\label{l1}
If $n\ge 1$, $y_n\leq x\leq y_{n+1}$, then $f(x)\in Q^{-}_n\cup T^{-}_n\cup
Q^{-}_{n+1}$.
\end{lemma}

\begin{proof}
We have
\begin{multline*}
f(x)=\sum_{k=1}^{n-2}\biggl(\frac{-ia_ky_k}{x+iy_k}+a_k\biggr)+
\sum_{k=n-1}^{k=n+2}\biggl(\frac{-ia_ky_k}{x+iy_k}+a_k\biggr)\\+
\int_0^x\sum_{k=n+3}^\infty\frac{ia_ky_k}{(t+iy_k)^2}\,dt=I_1+I_2+I_3.
\end{multline*}
Furthermore,
$$
I_1=w_{n-1}+\sum_{k=1}^{n-2}\frac{-ia_ky_k}{x+iy_k}=w_{n-1}+J_1,
$$
where
$$
|J_1|\le
\frac{1}{|x|}\sum_{k=1}^{n-2}|a_k|y_k\le
|a_n|\frac{(1-\varepsilon)^{-2}}{Q^2}\sum_{l=0}^\infty\biggl(\frac{(1-\varepsilon)^{-1}}{Q}\biggr)^l
\le\frac{2}{Q^2}|a_n|.
$$
On the other hand,
$$
|I_3|\le |x|\sum_{k=n+3}^\infty\frac{|a_k|}{y_k}\le\frac{2}{Q^2}|a_n|.
$$
It remains to estimate $I_2$. We have $x=\alpha y_n$ for some $\alpha\in[1,Q]$. Then
\begin{multline*}
I_2=-\frac{a_{n-1}iy_{n-1}}{x+iy_{n-1}}+a_{n-1}-\frac{a_{n}iy_{n}}{x+iy_{n}}+a_n\\
-\frac{a_{n+1}iy_{n+1}}{x+iy_{n+1}}+a_{n+1}-\frac{a_{n+2}iy_{n+2}}{x+iy_{n+2}}+a_{n+2}=\\=
a_n\Bigl(-\frac{i}{Q\alpha+i}+1 -\frac{i}{\alpha+i}+1-\frac{Qi}{\alpha+Qi}+1-\frac{Q^2i}{\alpha+Q^2i}+1\Bigr)
+R_n\\=J_2+R_n,
\end{multline*}
where $|R_n|\le 100\varepsilon|a_n|$. So,
$$
f(x)=w_{n+1}+a_n\biggl(-\frac{i}{Q\alpha+i}-\frac{i}{\alpha+i}+\frac{\alpha}{\alpha+Qi}+
\frac{\alpha}{\alpha+Q^2i}\biggr)+S_n,
$$
where
$$
|S_n|\le|a_n|(4Q^{-2}+100\varepsilon)\le 5|a_n|Q^{-2}.
$$
We conclude that $f(x)\in Q^{-}_n\cup T^{-}_n\cup Q^{-}_{n+1}$.
\end{proof}

\begin{lemma}\label{l2}
If $n\ge 1$, $y_n\le x\le y_{n+2}$, then
$\Re(\overline{a_n}f'(x))>0$.
\end{lemma}

\begin{proof}
We have
\begin{multline*}
f'(x)=H(x)=\sum_{k=1}^\infty\frac{ia_ny_n}{(x+iy_n)^2}=
\sum_{k=1}^{n-3}\frac{ia_ny_n}{(x+iy_n)^2}+\sum_{k=n+5}^\infty\frac{ia_ny_n}{(x+iy_n)^2}+\\
i\sum_{k=n-2}^{n+4}\frac{a_ky_k(x^2-y^2_k)}{(x^2+y^2_k)^2}+\sum_{k=n-2}^{n+4}\frac{2a_kxy^2_k}{(x^2+y_k^2)^2}
=\Sigma_1+\Sigma_2+\Sigma_3+\Sigma_4.
\end{multline*}
The assertion of the lemma is a direct consequence of the
following estimates:
\begin{gather*}
|\Sigma_1|\le\frac{1}{|x|^2}\sum_{k=1}^{n-3}|a_k|y_k\le\frac{2|a_n|}{Q^3|x|},\\
|\Sigma_2|\le\sum_{k=n+5}^\infty\frac{|a_k|}{y_k}\le\frac{2|a_n|}{Q^3|x|},\\
|\Re(\overline{a_n}\Sigma_3)|\le 1000\varepsilon|a_n|^2\frac{1}{|x|}\le\frac{|a_n|^2}{Q^3|x|},\\
\Re(\overline{a_n}\Sigma_4)\ge\frac{|a_n|^2}{Q^2|x|}.
\end{gather*}
\end{proof}

In a similar way we obtain

\begin{lemma} \label{l3}
\begin{enumerate}
\item[\rm(a)] If $n\ge 1$, $y_{-n-1}\le x\le y_{-n}$, then $f(x)\in Q^+_n\cup T^+_n\cup Q^+_{n+1}$.
\item[\rm(b)] If $n\ge 1$, $y_{-n-2}\le x\le y_{-n}$, then $\Re(\overline{a_n}f'(x))<0$.
\item[\rm(c)] If $0\le x\le y_1$, then $\Re(\overline{a_1}f'(x))>0$, $f(x)\in Q^-_0\cup T^-_0\cup Q^-_{1}$.
\item[\rm(d)] If $y_{-1}\le x\le 0$, then $\Re(\overline{a_1}f'(x))<0$, $f(x)\in Q^+_{-1}\cup T^+_{-1}\cup Q^+_{0}$.
\item[\rm(e)] If $n\ge 1$, $z\in\mathbb C_+$, and $|z|=y_n$, then $f(z)\in Q_n^+\cup Q_n^-$.
\end{enumerate}
\end{lemma}

Lemmas \ref{l1}--\ref{l3} together imply Proposition~\ref{mainprop}.
\end{proof}

\begin{proof}[Proof of Theorem~\ref{t2}]
The estimates in the proof of Lemma~\ref{l2} show that the function $f$ constructed in
Proposition~\ref{mainprop} satisfies the estimates
\begin{equation}
|f(x)-f(y)|\ge\frac{|x-y|}{Q^{25}(1+ |x|)^2},\qquad x,y\in[y_n,y_{n+2}],
\label{ft1}
\end{equation}
for $n\in\mathbb Z\setminus\{-2,-1,0\}$.
Furthermore,
\begin{equation*}
|f(x)-f(y)|\ge Q^{-25}|x-y|,\quad y_{-2}\le x\le y\le y_2.
\end{equation*}

For large $Q$ we can find a function $F_0$ in $\mathcal{P}W^{\infty}_{[0,\pi]}$ such that
\begin{equation}
|F_0(x)|\le Q^{-1},\qquad |F'_0(x)|\le \frac{1}{Q^{30}(1+|x|)^2}, \qquad x\in\mathbb R,
\label{nee}
\end{equation}
and $F_0(-iQ^n)=1$ as $n\ge1$.
Indeed, denote
\begin{gather*}
S(z)=\prod_{n\ge 1}\Bigl(1+\frac{iz}{Q^n}\Bigr),\\
R(z)=e^{i(\pi/2)z}\sin\Bigl(\frac{z}2\Bigr)\cdot \prod_{n\ge 1}\Bigl(1-\frac{z^2}{4\pi^2Q^{2n}}\Bigr)^{-1}.
\end{gather*}
Then
\begin{gather*}
\log|S(z)|\sim\frac{(\log(2+|z|))^2}{2\log Q}, \qquad \dist(z,\{-iQ^n\})>1,\\
\log|S'(x)|\sim\frac{(\log x)^2}{2\log Q},\qquad |x|\to\infty,\\
\log|S'(-iQ^n)|\sim\frac{n^2}2\log Q,\qquad n\ge 1,\\
\log|R(-iQ^n)|\sim \pi Q^n,\qquad n\ge 1,\\
\max\bigl(\log|R(x)|,\log|R'(x)|\bigr)\le O(1)-\frac{(\log(2+|x|))^2}{\log Q},\qquad x\in\mathbb R,
\end{gather*}
where $A(u)\sim B(u)$ means that $\lim_{u\to\infty}A(u)/B(u)=1$.
It remains to set
$$
F_0(z)=R(z)\cdot\sum_{n\ge 1}\frac{i}{Q^nS'(-iQ^n)R(-iQ^n)}\cdot\frac{S(z)}{1+izQ^{-n}}.
$$
Estimate \eqref{nee} holds for sufficiently large $Q$.

\begin{lemma} \label{l4}
Let $n\ge 1$. If $y_n\le x\le y_{n+1}$, then $f(x)(1-F_0(x))\in Q^{-}_n\cup T^{-}_n\cup
Q^{-}_{n+1}$. If $y_{-n-1}\le x\le y_{-n}$, then $f(x)(1-F_0(x))\in Q^{+}_n\cup T^{+}_n\cup
Q^{+}_{n+1}$. Finally, if $y_{-1}\le x\le y_1$, then $f(x)(1-F_0(x))\in
Q^-_0\cup T^-_0\cup Q^-_{1}\cup Q^+_{-1}\cup T^+_{-1}\cup Q^+_{0}$.
\end{lemma}

\begin{proof}
We just use the estimate $|f(x)F_0(x)|\leq y^{-3}_n$ and the argument from the proof of Lemma~\ref{l1} to
get the result.
\end{proof}

\begin{lemma} \label{l8}
The function $F=f(1-F_0)$ is univalent in $\mathbb C_+$.
\end{lemma}

\begin{proof} It suffices to verify that $F$ is injective on $\mathbb R$.
If $F(x)=F(y)$, $x<y$, then, by Lemma~\ref{l4}, $y_n\le x<y\le y_{n+2}$
for some $n$. Since $F(x)=F(y)$, we have
$$
f(x)-f(y)=F_0(x)(f(x)-f(y))+f(y)(F_0(x)-F_0(y))=K_1+K_2.
$$
If $n\in\mathbb Z\setminus\{-2,-1,0\}$, then
\begin{gather*}
|K_1|\le |f(x)-f(y)|/2,\\
|K_2|\le |F_0(x)-F_0(y)|\le |x-y|\cdot|F'(\zeta)|,
\end{gather*}
for some $\zeta\in[y_n,y_{n+2}]$. Therefore,
$$
|K_2|\leq \frac{|x-y|}{Q^{30}(1+|\zeta|)^2},
$$
and we obtain a contradiction to \eqref{ft1}.

An analogous argument works for $y_{-2}\le x\le y\le y_2$.
\end{proof}

Finally, $F\in\mathcal{P}W^{\infty}_{[0,\pi]}$ and  $\dim_H(\partial F(\mathbb C_+))=\dim_H(\mathcal H_\infty)$ could be any number in the interval $[1,2]$.
\end{proof}

\section{Proof of Theorem~\ref{t0-2}}
\label{section5}

As mentioned after the statement of the theorem, we deal here just with the lower estimate.
It suffices to show that for some absolute constant $\alpha>0$ and for every integer
$N\ge 1$ there exists a rational function $f$ of
degree $N$ univalent in $\mathbb C_+$ and such that
$$
\int_{\mathbb R}|f'(x)|\,dx>\alpha\sqrt{N}\|f\|_{\infty,\mathbb C_+}.
$$

To find such a function we use the construction in
Proposition~\ref{mainprop} with finite number of points $w_n$. For $\beta>0$ set
$$
w_n=\biggl(1-\frac{n}{2N}\biggr)\exp(2\pi i\cdot \beta nN^{-1/2}),
\qquad 1\le n\le N-1.
$$
Direct calculations show that for sufficiently small $\beta=\beta_0$ the sequence $w_n$
satisfies all the properties necessary to proceed with the argument in Proposition~\ref{mainprop}. Finally,
$$
\int_{\mathbb R}|f'(x)|\,dx\gtrsim \sum_{1\le n< N-1}|w_{n+1}-w_n|\ge \tau\sqrt{N}
$$
and $\|f\|_{\infty,\mathbb{D}}\le C$ for some absolute $\tau>0$ and $C<\infty$ that completes the proof.

\end{document}